\documentclass[11pt]{amsart}

\usepackage[english]{babel}
\usepackage[utf8]{inputenc}
\usepackage{amsmath,amssymb,amsfonts}
\usepackage[nocompress]{cite}
\usepackage{microtype}
\DisableLigatures[f]{encoding = *, family = * }
\usepackage{tabularx}
\usepackage{enumerate}
\usepackage{graphicx}
\usepackage{tikz}
\usepackage{pgfplots}
\usepackage{pgfplotstable}
\usetikzlibrary{through,calc,arrows}
\usepackage{float}
\usepackage{bm}
\usepackage{color}
\usepackage{verbatim}
\usepackage{dcolumn}
\usepackage{verbatim,booktabs}
\usepackage[top=1in, bottom=1in, left=0.7in, right=0.7in]{geometry}
\usepackage[pdftitle={A discrete-ordinate discontinuous-streamline diffusion method for the radiative transfer equation},
            pdftex, unicode=true,
            colorlinks,
            linkcolor=red,
            anchorcolor=blue,
            citecolor=green
            ]{hyperref}

\makeatletter
\@addtoreset{equation}{section}

\makeatother
\newtheorem{thm}{Theorem}[section]
\newtheorem{lem}[thm]{Lemma}

\newtheorem{remark}[thm]{Remark}

\title[A discrete-ordinate discontinuous-streamline diffusion method for RTE]
{A discrete-ordinate discontinuous-streamline diffusion method for the radiative transfer equation}
\author[C. Wang, Q. Sheng, W. Han]{Cheng Wang$^\dag$ \and Qiwei Sheng$^\ddag$ \and Weimin Han$^\S$}
\date{}

\begin{document}

\thanks{$^\dag$ Department of Mathematics, Tongji University, Shanghai 200092, China (wangcheng@tongji.edu.cn).}
\thanks{$^\ddag$ Computational and Applied Mathematics Group, Oak Ridge National Laboratory, Oak Ridge, TN 37831 (shengq@ornl.gov)}
\thanks{$^\S$ Department of Mathematics, University of Iowa, Iowa City, IA 52242 (weimin-han@uiowa.edu)}

\begin{abstract}
The radiative transfer equation (RTE) arises in many different areas of science and engineering.
In this paper, we propose and investigate a discrete-ordinate discontinuous-streamline diffusion
(DODSD) method  for solving the RTE, which is a combination of the discrete-ordinate technique and the discontinuous-streamline
diffusion method.  Different from the discrete-ordinate discontinuous Galerkin (DODG) method for the RTE, an artificial
diffusion parameter is added to the test functions in the spatial discretization.  Stability and error estimates
in certain norms are proved.  Numerical results show that the proposed method can lead to a more accurate approximation
in comparison with the DODG method.



\end{abstract}

\maketitle

\section{Introduction}\label{sec:intro}
The radiative transfer equation, which describes the scattering and absorbing of radiation
through a medium, plays an important role in a wide range of applications such as
astrophysics, atmosphere and ocean,
heat transfer, neutron transport and nuclear physics, and so on.
Today, research on the RTE remains to be very active and important, especially in the biomedical optics fields, see e.g. \cite{A2009,Bal09,HECW11,Ren10,THH13}.

The RTE can be viewed as a hyperbolic type integro-differential equation. Due to the
involvement of both integration and differentiation in the equation, as well as the
high dimension of the problem, it is challenging to develop effective numerical methods
for solving the RTE. The numerical methods can be basically divided into two categories:
statistical schemes and deterministic schemes. The interested readers are referred to
\cite{FC1971,H1998,book:LK1991,SG2008,ES2012,han2013numerical,GZ2013,HES2013,McClarren20102290}.

The discrete-ordinate (DO) method \cite{C1955,KKWV1995,LM2010}, also called the $S_N$ method,
is the most popular deterministic method for the RTE, owing to the good compromise among accuracy, flexibility,
and moderate computational requirements.
This method solves the radiative transfer equation along a discrete set of angular directions,
which are the nodal points of a numerical quadrature approximating the integral term on the unit sphere,
thus reducing the RTE to a semi-discretized first-order hyperbolic system.
To solve the semi-discretized hyperbolic system, it is natural to use the discontinuous
Galerkin (DG) discretization, leading to the so-called discrete-ordinate
discontinuous Galerkin method. In \cite{HHE2010}, a DODG method was proposed for the RTE, and
error estimates in certain discrete norms were obtained.

The object of this paper is to propose and investigate a discrete-ordinate discontinuous-streamline
diffusion method for solving the RTE. Such a method is a  combination of the
discrete-ordinate technique and the discontinuous-streamline diffusion (DSD) method.
The streamline diffusion (SD) finite element method was proposed by
Hughes et al. \cite{Hughes1986341} and Johnson et al. \cite{Johnson198199} in order to cope with the usual instabilities
caused by the convection term for the convection–dominated problem.
In \cite{AK2004,AK2013}, the streamline diffusion finite element method was analyzed  for the multi-dimensional Vlasov-FokkerPlanck system and  Fermi pencil beam equation.
The DSD method keeps the fundamental structure of the DG method while replacing the Galerkin elements by the SD framework in the upwind iteration procedure.
In \cite{SunTanWu1998}, the DSD method was employed successfully in solving first order hyperbolic problems,
where such a modification preserves the advantages of both the upwind approach and the DG method, and also
further improves the stability.
In this contribution, we seek to improve the DG method for RTE by employing the DSD scheme and derive
error estimates of the DODSD method in a norm including the directional gradient. While the DSD approach
has been developed and applied to hyperbolic systems or convection–dominated problems, this paper represents
the first attempt, to our knowledge, to construct DSD schemes for the RTE.  Our numerical results show
that the DODSD method can lead to a more accurate solution in comparison with the DODG method.

The rest of this paper is organized as follows. In Section \ref{sec:RTE}, we introduce the RTE
and recall a few basic related results. In Section \ref{sec:NM}, we derive the discrete-ordinate
discontinuous-streamline diffusion method, and in Section \ref{sec:NA} we present a stability
and convergence analysis for the proposed method. Numerical examples are presented in
Section \ref{sec:NE}, illustrating the performance of the numerical method and providing
numerical evidence of the theoretical error estimates. Finally, a few concluding remarks
are given in Section \ref{sec:conclusion}.

Throughout this paper, standard notation is used for Sobolev spaces, and the corresponding
semi-norms and norms \cite{book:C}. Moreover, the letter $C$ denotes a generic positive
constant whose value may be different at different occurrences.

\section{Radiative transfer equation}\label{sec:RTE}

Let $X$ be a bounded domain in $\mathbb{R}^d$ $(d=2,3)$ with a smooth boundary $\partial X$.
Denote by $\bm{n}(\bm{x})$ the unit outward normal for  $\bm{x}\in \partial X$.
Let  $\Omega$ be the angular space, i.e., the unit circle in $\mathbb{R}^2$, or the unit sphere
in $\mathbb{R}^3$. For each fixed direction $\bm{\omega}\in \Omega$, we introduce the
following subsets of $\partial X$:
\[ \partial X_{{\bm{\omega}},-}=\{\bm{x} \in{\partial X}\colon\bm{\omega}\cdot \bm{n}(x) <0 \},\quad
  \partial X_{{\bm{\omega}},+}=\{\bm{x} \in{\partial X}\colon\bm{\omega}\cdot \bm{n}(x) \ge 0 \}.\]
Then, we define
\[  \Gamma_-=\{(\bm{x},{\bm{\omega}})\colon \bm{x}\in \partial X_{{\bm{\omega}},-},
  {\bm{\omega}}\in \Omega\}, \quad \Gamma_+=\{(\bm{x},{\bm{\omega}})\colon
  \bm{x}\in\partial X_{{\bm{\omega}},+},{\bm{\omega}}\in\Omega\} \]
as the incoming and outgoing boundaries.

We define the integral operator $S$ by
\[ (Su)(\bm{x},{\bm{\omega}})=\int_\Omega g(\bm{x},{\bm{\omega}}\cdot
 \hat{{\bm{\omega}}}) u(\bm{x},\hat{{\bm{\omega}}})d\sigma(\hat{{\bm{\omega}}}),\]
where $g$ is a nonnegative normalized phase function satisfying
\begin{equation}\label{normilized}
  \int_\Omega g(\bm{x},{\bm{\omega}}\cdot\hat{{\bm{\omega}}}) d \sigma(\hat{{\bm{\omega}}}) = 1
\quad\forall\, \bm{x}\in X,\,{\bm{\omega}}\in \Omega.
\end{equation}
In most applications, the function $g$ is independent of $\bm{x}$. As an example,
a commonly used phase function is the following Henyey-Greenstein (H-G) function:
\begin{equation}\label{HG}
  g(t)=
  \left\{
  \begin{array}{ll}
    \frac{1-\eta^2}{2\pi(1+\eta^2-2\eta t)}\quad &d=2, \\
  \frac{1-\eta^2}{4\pi(1+\eta^2-2\eta t)^{3/2}}\quad &d=3 ,
  \end{array}
  \right.
\end{equation}
where the parameter $\eta\in (-1,1)$ is the anisotropy factor of the scattering medium. Note that
$\eta=0$ for isotropic scattering, $\eta>0$ for forward scattering, and $\eta <0$ for backward scattering.

With the above notation, a boundary value problem of the radiative transfer equation (RTE) reads
\begin{align}\label{eq:rte}
{\bm{\omega}}\cdot \nabla u(\bm{x},{\bm{\omega}})+\sigma_t(\bm{x})u(\bm{x},{\bm{\omega}})
&=\sigma_s(\bm{x})(Su)(\bm{x},{\bm{\omega}})+f(\bm{x},{\bm{\omega}}),
\quad (\bm{x},{\bm{\omega}})\in X\times \Omega,\\
u(\bm{x},{\bm{\omega}})&=0,\quad \hspace{4.3cm}(\bm{x},{\bm{\omega}})\in \Gamma_-.\label{eq:rtee}
\end{align}
Here $\sigma_t=\sigma_a+\sigma_s$, $\sigma_a$ is the macroscopic absorption cross section,
$\sigma_s$ is the macroscopic scattering cross section, and $f$ is a source function.
We assume these given functions have the properties that
\begin{eqnarray}\label{as:1}
  \sigma_t,\sigma_s\in L^\infty(X),\ \sigma_s \ge 0 \ \mbox{a.e.\ in}\ X, \
  \text{and}\ \sigma_t-\sigma_s \ge c_0\ \mbox{in}\ X\ \mbox{for a constant}\ c_0 >0,\\
  f(\bm{x},{\bm{\omega}})\in L^2(X\times \Omega)
  \mbox{ and is a continuous function with respect to }{\bm{\omega}}\in \Omega.
\end{eqnarray}

It is shown in \cite{agoshkov1998boundary} that the problem \eqref{eq:rte}--\eqref{eq:rtee}
has a unique solution $u\in H^1_2(X\times \Omega)$, where
\[  H^1_2(X\times\Omega):=\{v \in L^2(X\times \Omega)\colon
 {\bm{\omega}}\cdot \nabla v \in L^2(X\times \Omega)\} \]
with ${\bm{\omega}}\cdot \nabla v$ denoting the generalized directional derivative of $v$
in the direction ${\bm{\omega}}$.

\section{A discrete-ordinate discontinuous-streamline diffusion method}\label{sec:NM}

In this section, a discrete-ordinate discontinuous-streamline diffusion method  is presented
for solving the radiative transfer problem \eqref{eq:rte}--\eqref{eq:rtee}. The numerical
scheme is formed in two steps: 
First, we use the discrete-ordinate method to approximate the integral term in the RTE,
resulting in a system of linear hyperbolic partial differential equations. Then these coupled linear hyperbolic
equations are further discretized by the discontinuous-streamline diffusion method.

\subsection{Angular discretization}

To approximate the integration term $Su$, we employ a numerical quadrature of the form
\begin{equation}
\int_\Omega F(\bm{\omega})d \sigma(\bm{\omega})\approx \sum_{l=0}^Lw_l F(\bm{\omega}_l),
 \quad w_l>0,\ \bm{\omega}_l\in \Omega,\ 0\le l\le L, \label{quad0}
\end{equation}
where $F$ is a continuous function over the unit sphere $\Omega$.

\subsubsection{Quadrature scheme in the two-dimensional (2D) domain}
Introduce the spherical coordinate system
\begin{equation}\label{coor:new}
  {\bm{\omega}}=(\cos\theta,\sin\theta)^T,\quad   0\le \theta\le 2\pi.
\end{equation}
Noting that $d\sigma({\bm{\omega}})=d\theta$ holds for the coordinate system \eqref{coor:new},
we have
\[ \int_\Omega F({\bm{\omega}})d\sigma({\bm{\omega}})=\int_0^{2\pi}\bar{F}(\theta)\,d\theta,\]
where $\bar{F}$ stands for the representation of $F$ in the spherical coordinates.

One possible quadrature scheme for the above integral is the composite trapezoidal formula
\begin{eqnarray}\label{ctf}
  \int_0^{2\pi}\bar{F}(\theta)\,d\theta \approx \frac{h_\theta}{2}\left(\bar{F}(\theta_0)
 +\sum_{i=1}^{L-1}2\bar{F}(\theta_i)+\bar{F}(\theta_L)\right):=\sum_{i=0}^L w_i \bar{F}(\theta_i),
\end{eqnarray}
where $\{\theta_i\}$ are evenly spaced on $[0,2\pi]$ with a spacing $h_\theta=2\pi/L$, i.e.,
$\theta_i=ih_\theta$, $w_0=w_L=\frac{h_\theta}{2}$, and $w_i=h_\theta$ for $1\le i\le L-1$.
It is known that (see, e.g.~\cite{Atkinson1989})
\begin{equation}\label{neq:int}
  \int_0^{2\pi}\bar{F}(\theta)\,d\theta - \frac{h_\theta}{2}\left(\bar{F}(\theta_0)
  +\sum_{i=1}^{L-1}2\bar{F}(\theta_i)+\bar{F}(\theta_L)\right)
   =-\frac{\pi h_\theta^2}{6} \bar{F}''(\theta).
\end{equation}

\subsubsection{Quadrature scheme in the three-dimensional (3D) domain}
Introduce the spherical coordinate system
\begin{equation}\label{coor:new1}
{\bm{\omega}}=(\sin\theta\cos\psi,\sin\theta\sin\psi,\cos\theta)^T,\quad
  0\le \theta\le \pi,\ 0\le \psi\le 2\pi.
\end{equation}
Then we have $d\sigma({\bm{\omega}})=\sin\theta d\theta d\psi $.
By using the spherical coordinate system \eqref{coor:new1}, we obtain
\[  \int_\Omega F({\bm{\omega}})\,d\sigma({\bm{\omega}})
 =\int_0^{2\pi}\!\!\!\int_0^{\pi}\bar{F}(\theta,\psi)\sin\theta d\theta d\psi.\]

One family of quadratures for the above integral is given by the product numerical
integration formulas. For example,
\begin{equation} \label{eq:int_3d}
  \int_\Omega F(\bm{\omega})d \sigma(\bm{\omega})\approx
 \frac\pi m \sum_{j=1}^{2m}\sum_{i=1}^m \overline{w}_i\bar{F}(\theta_i,\psi_j),
\end{equation}
where $\{\theta_i\}$ are chosen so that $\{\cos\theta_i\}$ and $\{\overline{w}_i\}$ are the
Gauss-Legendre nodes and weights on $[-1,1]$. The points $\{\phi_j\}$ are evenly spaced
on $[0,2\pi]$ with a spacing of $\pi/m$.
Regarding the accuracy of the quadrature \eqref{eq:int_3d}, we have (see, e.g.~\cite{HesseSloan2006})
\begin{equation}\label{neq:int1}
  \left|\int_\Omega F(\bm{\omega})\,d\sigma(\bm{\omega}) - \sum_{l=0}^Lw_l F(\bm{\omega}_l)\right|
  \le c_s n^{-s}\|F\|_{s,\Omega} \quad \forall\, F\in H^s(\Omega),\ s>1,
\end{equation}
where $c_s$ is a positive constant depending only on $s$, and $n$ denotes the degree of
precision of the quadrature.

\subsubsection{Discrete-ordinate method}

Based on the numerical quadrature \eqref{quad0}, the integral operator $S$ is approximated
by a discretized operator $S_d$ given by
\begin{equation}
  S_du(\bm{x},{\bm{\omega}})=\sum_{i=0}^L w_i g(\bm{x},{\bm{\omega}}
 \cdot {\bm{\omega}}_i) u(\bm{x},{\bm{\omega}}_i).
\end{equation}

For later analysis, we define
\begin{equation}\label{def:m}
  m(\bm{x})=\max_{0\le l\le L}\sum_{i=0}^L w_i g(\bm{x},{\bm{\omega}}_l\cdot {\bm{\omega}}_i).
\end{equation}
In the 2D case, if $g(\bm{x},t)$ is continuous in $\bm{x}\in\overline{X}$ and twice
continuously differentiable with respect to $t\in [-1,1]$, then we get from \eqref{neq:int}
and \eqref{normilized} that
\begin{equation}
  \left|1-\sum_{i=0}^Lw_ig(\bm{x},{\bm{\omega}}_l\cdot{\bm{\omega}}_i)\right|\le O( h_\theta^2).
\end{equation}
This implies
\begin{equation}
  \|m(\bm{x})\|_{0,\infty,X}\le 1+O(h_\theta^2).
\end{equation}
Therefore, for $h_\theta$ sufficiently small, there exists a positive constant $c'_0 $ satisfying
\begin{equation}\label{eq:312}
  \sigma_t-m(\bm{x})\sigma_s\ge\sigma_t-\sigma_s-O(h_\theta^2)\sigma_s
\ge c_0-O(h_\theta^2)\sigma_s \ge c'_0\quad \forall \,\bm{x}\in X.
\end{equation}

In the 3D case, if $g(\bm{x},{\bm{\omega}}_l\cdot)$ is an $H^s(\Omega)$ $(s>1)$ function for
any fixed $\bm{x}\in X$ and ${\bm{\omega}}_l\in \Omega$, then we get from (\ref{neq:int1})
and (\ref{normilized}) that
\begin{equation}\label{eq:36}
  \left|1- \sum_{i=1}^Lw_l g({\bm{\omega}}_l\cdot {\bm{\omega}}_i)\right|
  \le c_s n^{-s}\|g({\bm{\omega}}_l\cdot)\|_{s,\Omega}.
\end{equation}
This also implies that $\|m(\bm{x})\|_{0,\infty,X}\approx 1$ and \eqref{eq:312} holds
in the 3D case when a high-order quadrature rule is used.

\begin{remark}
Numerical tests are provided in \cite{HHE2010} to demonstrate that \eqref{eq:36} holds
for the Henyey-Greenstein phase function \eqref{HG}.
\end{remark}

Using the operator $S_d$, we can discretize the radiative transfer equation
\eqref{eq:rte}--\eqref{eq:rtee} in each angular direction $\bm{\omega}_l$  to get
\begin{equation}\label{eq:DO}
{\bm{\omega}}_l\cdot \nabla u^l +\sigma_tu^l=\sigma_s\sum_{i=0}^L w_i g(\cdot,{\bm{\omega}}_l
\cdot{\bm{\omega}}_i)u^i+f_l\ {\quad\rm in}\ X, \quad u^l=0\ {\rm on}\ \partial_-^l X,\quad 0\le l\le L,
\end{equation}
where $f_l=f(\bm{x},{\bm{\omega}}_l)$ and $u^l=u^l(\bm{x})$ is an approximation of
$u(\bm{x},{\bm{\omega}}_l)$. Here and below, we use the simplified notation
$\partial_\pm^l X:=\partial X_{{\bm{\omega}}_l,\pm}$.

\begin{remark}
Note that the Henyey-Greenstein function \eqref{HG} is smooth for $\eta < 1$.
Formally, $\eta = 1$ corresponds to the case where there is no scattering
among different directions and $(Su)(\bm{x},{\bm{\omega}}) =u(\bm{x},{\bm{\omega}})$.
As a result, the system \eqref{eq:DO} is reduced to a set of uncoupled first order
transfer equations, which can be solved easily, and the analysis is the same
as that for a single transfer equation.
\end{remark}

\subsection{Spatial discretization}
After the angular discretization, the RTE is reduced to a  system of first-order hyperbolic
partial differential equations in space. Now we discretize \eqref{eq:DO} by the
discontinuous-streamline diffusion method.

Let $\{T_h\}_h$ be a regular family of finite element partitions of $X$, $h$ being the mesh size parameter.
Denote by $\bm{n}_K$ the unit outward normal to $\partial K$ for $K\in T_h$. Let $E_h^i$
be the set of all interior boundaries (faces for $d=3$ or edges for $d=2$) of $T_h$. For any positive integer $k$,
let $P_k(K)$ be the set of all polynomials on $K$ of a total degree no more than $k$.

For a fixed direction $\bm{\omega}_l$, we define the incoming and outgoing boundaries of $K\in T_h$ by
\[  \partial^l_- K=\{\bm{x}\in \partial K\colon {\bm{\omega}_l}\cdot \bm{n}(\bm{x})<0\},\quad
  \partial^l_+ K=\{\bm{x}\in \partial K\colon {\bm{\omega}_l}\cdot \bm{n}(\bm{x})\ge0\}.\]
We remark that each edge of  an element $K\in T_h$ is either an incoming boundary or an outgoing boundary.

Let $K^l_+$ and $K^l_-$ be two adjacent elements sharing  $e\in E_h^i$, where the
normal direction $\bm{n}_e^l$ pointing from $K^l_-$ to $K^l_+$ satisfies
$\bm{\omega}\cdot \bm{n}_e^l \ge 0$ (cf.\ Figure \ref{fig02}).
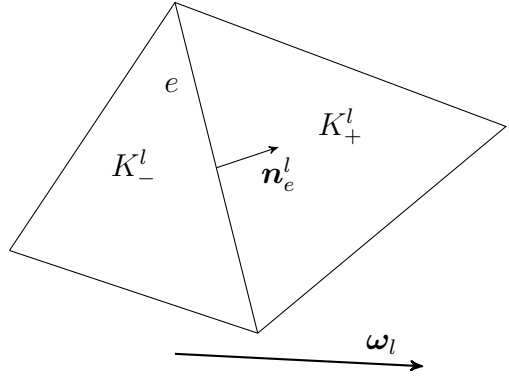
\begin{figure}[htbp]
\begin{center}
  \begin{tikzpicture}[thin,scale=.55]
    \draw (0,0) coordinate (A)
        -- (4,6) coordinate (B)
        -- (12,3) coordinate (C)
        -- (6,-2) coordinate (D);
    \draw (B) -- (D);
    \draw (D) -- (A);
    \draw[->, >=stealth'] (5, 2) -- (6.5,2.5) node[below]{\large $\bm{n}^l_e$};
    \node at (3, 2) {\large $K^l_-$};
    \node at (8,3) {\large $K^l_+$};
    \node[right] at (3.5,4) {\large $e$};
    \draw[thick, ->, >=stealth'] (4,-2.5)--(10,-2.8);
    \node[below] at (9,-1.8){\large $\bm{\omega}_l$};
  \end{tikzpicture}
\end{center}
\caption{An example of $K^l_-$, $K^l_+$, and $n_e^l$  in 2D}\label{fig02}
\end{figure}
For a scalar-valued function $v$, we define
\[ v^l_+=v|_{K^l_+},\quad v^l_-=v|_{K^l_-}, \text{ and } [v^l]=v^l_+-v^l_-\ \text{ on } e.\]
For any domain $D\subseteq X$  with  boundary $\partial D$ (resp.\ $\partial_\pm^l D$) ,
let $(\cdot,\cdot)_D$ and $\langle\cdot,\cdot\rangle_{\partial D}$
(resp.\ $\langle \cdot,\cdot\rangle_{\partial^l_\pm D}$) be the $L^2$ inner product on $D$ and on
$\partial D$  (resp.\ $\partial^l_\pm D$).

Using the above notation, the DODG method, which  has been developed in \cite{HHE2010},
is to find $u^l_h \in P_k(K)$ such that for any $K\in T_h$, $0\le l \le L$,
\begin{align}\label{DODG}
&\left({\bm{\omega}}_l\cdot \nabla u^l_h +\sigma_tu_h^l,v_h^l\right)_K+\left<{ [u_h^l]},v^l_
+|{\bm{\omega}_l}\cdot \bm{n}|\right>_{\partial^l_- K}\\
&\qquad \qquad= \left(\sigma_s\sum_{i=0}^L w_i g(\cdot,{\bm{\omega}}_l\cdot{\bm{\omega}}_i)u_h^i
+f_l,v_h^l\right)_K\quad \forall\, v_h^l \in P_k(K)\nonumber
\end{align}
with
\begin{equation}\label{eq:3.21}
  u^l_-=0\quad{\rm on }\ \partial^l_- K\subset \partial_-^l X.
\end{equation}

We now replace the Galerkin elements in the above DODG formulation \eqref{DODG} by the SD framework,
and add an artificial diffusion term in the test function.
Then the discrete-ordinate discontinuous-streamline diffusion
(DODSD) method can be described as follows:
to find $u^l_h \in P_k(K)$ such that for any $K\in T_h$, $0\le l \le L$,
\begin{align}\label{DODSD}
&\left({\bm{\omega}}_l\cdot \nabla u^l_h +\sigma_tu_h^l,v_h^l+\delta\, {\bm{\omega}_l}
\cdot \nabla v_h^l\right)_K+\left<{ [u_h^l]},v^l_
+|{\bm{\omega}_l}\cdot \bm{n}|\right>_{\partial^l_- K}\\
&\qquad \qquad= \left(\sigma_s\sum_{i=0}^L w_i g(\cdot,{\bm{\omega}}_l\cdot{\bm{\omega}}_i)u_h^i
+f_l,v_h^l+\delta\, {\bm{\omega}_l}\cdot \nabla v_h^l\right)_K\quad \forall\, v_h^l \in P_k(K)\nonumber
\end{align}
with
\begin{equation}
  u^l_-=0\quad{\rm on }\ \partial^l_- K\subset \partial_-^l X.
\end{equation}
Here  $\delta=\bar{c}\,h$ is an artificial diffusion parameter with some $\bar{c}>0$ and
$v^l_\pm:=(v_h^l)_\pm$.

Obviously, the DODG method is the special case of the DODSD method with $\delta = 0$.
The effect of adding the diffusion parameter will be analyzed in the next section, and illustrated
by some numerical results in Section \ref{sec:NR}.

\section{Error analysis}\label{sec:NA}

In order to analyze the proposed DODSD method, we first present the  global formulation of
the discrete method \eqref{DODSD}--\eqref{eq:3.21}.
Associated with a direction $\bm{\omega}_l$, we define
\begin{align}\label{space:41}
  V_h^l&=\{v\in L^2(X)\colon v|_K\in P_k(K)\ \forall\, K\in T_h\},\\
W_h^l&=\left\{w \in L^2(X)\colon w|_K \in C(K)\cap H^1(K)\ \forall\, K\in T_h\right\}.
\end{align}
Letting $\bm{V}_h=\left(V_h^l\right)^{L+1}$ and $\bm{W}_h:=\left(W_h^l\right)^{L+1}$,
we have $\bm{V}_h\subset\bm{W}_h$. A generic element in $\bm{V}_h$ will be denoted by
$\bm{v}_h:=\{v_h^l\}_{l=0}^L$ or simply $\bm{v}_h:=\{v_h^l\}$.

The global formulation of the DODSD method \eqref{DODSD}--\eqref{eq:3.21} is then expressed as:
Find $\{u^l_h\}\in \bm{V}_h$ such that
\begin{align}\label{DODSD_Global}
 & \sum_{l=0}^Lw_l\sum_{K\in T_h}\left({\bm{\omega}}_l\cdot \nabla u_h^l
+\sigma_tu_h^l,v_h^l+\delta\, {\bm{\omega}_l}\cdot \nabla v_h^l\right)_K
+\sum_{l=0}^Lw_l\sum_{K\in T_h}\left<{ [u^l_h]},
v^l_+|{\bm{\omega}_l}\cdot \bm{n}|\right>_{\partial^l_- K}\\
&\qquad =  \sum_{l=0}^Lw_l\sum_{K\in T_h}\left(\sigma_s\sum_{i=0}^L
w_i g(\cdot,{\bm{\omega}}_l\cdot{\bm{\omega}}_i)u^i_h+f_l,v_h^l
+\delta\, {\bm{\omega}_l}\cdot \nabla v_h^l\right)_K\quad\forall \{v_h^l\}\in \bm{V}_h\nonumber
\end{align}
with
\begin{equation}\label{uboundary}
u^l_-=0\quad \mbox{on}\ \partial^l_- K \subset \partial^l_- X,\ 0\le l\le L.
\end{equation}

We define a bilinear form $a_h:\bm{W}_h\times \bm{W}_h\rightarrow \mathbb{R}$ as
\begin{align*}
  a_h(\bm{u}_h,\bm{v}_h)&=\sum_{l=0}^Lw_l\sum_{K\in T_h}\left({\bm{\omega}}_l\cdot \nabla u_h^l
  +\sigma_tu_h^l,v_h^l+\delta {\bm{\omega}_l}\cdot \nabla v_h^l\right)_K\\
&{}\quad+\sum_{l=0}^Lw_l\sum_{K\in T_h}\left<{ [u^l_h]},v^l_+|{\bm{\omega}_l}
 \cdot \bm{n}|\right>_{\partial^l_- K}\\
&{}\quad -  \sum_{l=0}^Lw_l\sum_{K\in T_h}\left(\sigma_s\sum_{i=0}^L w_i g(\cdot,
{\bm{\omega}}_l\cdot{\bm{\omega}}_i)u_h^i,v_h^l+\delta \,{\bm{\omega}_l}\cdot \nabla v_h^l\right)_K
\end{align*}
and a linear form $f:\bm{W}_h\rightarrow \mathbb{R}$ by
\[ f(\bm{v}_h)=\sum_{l=0}^Lw_l\sum_{K\in T_h}\left(f_l,v_h^l+\delta\, {\bm{\omega}_l}
\cdot \nabla v_h^l\right)_K.\]

Then we rewrite the DODSD method for the problem \eqref{eq:rte}--\eqref{eq:rtee}:
Find $\bm{u_h} \in \bm{V}_h$ such that
\begin{equation}\label{DODSD:global2b}
  a_h(\bm{u}_h,\bm{v}_h) = f(\bm{v}_h)\quad \forall\, \bm{v}_h \in \bm{V}_h,
\end{equation}
with
\begin{equation}\label{DODSD:global2e}
[u^l_h]=u^l_+\quad \mbox{on}\ \partial^l_- K \subset \partial^l_- X,\ 0\le l\le L.
\end{equation}

\subsection{Stability and unique solvability}

We begin with a useful lemma.

\begin{lem}\label{lem:31}
  For any $\bm{v}_h=\{v^l_h\}$,  $\bm{w}_h=\{w^l_h\}\in \left(L^2(\Omega)\right)^{L+1}$, we have
\begin{align*}
& \sum_{l=0}^Lw_l\left(\sigma_s\sum_{i=0}^L w_i g(\cdot,{\bm{\omega}}_l
\cdot{\bm{\omega}}_i)v^i_h,w^l_h\right)_X\\
&\qquad \le \left[\sum_{l=0}^Lw_l\left(m\sigma_sv^l_h,v^l_h\right)_X\right]^\frac12
    \left[\sum_{l=0}^Lw_l\left(m\sigma_sw^l_h,w^l_h\right)_X\right]^\frac12.
\end{align*}
\end{lem}
\begin{proof}
Interchanging  the order of summation, we have
\[  I:=\sum_{l=0}^Lw_l\left(\sigma_s\sum_{i=0}^L w_i g(\cdot,{\bm{\omega}}_l
\cdot{\bm{\omega}}_i)v_h^i,w_h^l\right)_X
=\sum_{i=0}^Lw_i\sum_{l=0}^L\left(\sigma_s w_l g(\cdot,{\bm{\omega}}_l
\cdot{\bm{\omega}}_i)v_h^i,w_h^l\right)_X. \]
Using the Cauchy-Schwarz inequality, we get
\begin{equation}\label{eq:lem411}
I\le \sum_{i=0}^Lw_i  \left[\sum_{l=0}^L\left(\sigma_s w_l g(\cdot,{\bm{\omega}}_l
\cdot{\bm{\omega}}_i)v_h^i,v_h^i\right)_X\right]^\frac12  \left[\sum_{l=0}^L\left(\sigma_s w_l
 g(\cdot,{\bm{\omega}}_l\cdot{\bm{\omega}}_i)w_h^l,w_h^l\right)_X\right]^\frac12.
\end{equation}
It follows from the definition \eqref{def:m} that
\[ \sum_{l=0}^L\left(\sigma_s w_l g(\cdot,{\bm{\omega}}_l\cdot{\bm{\omega}}_i)v_h^i,v_h^i\right)_X
  \le  \left(m\sigma_sv_h^i,v_h^i\right)_X. \]
Therefore, a combination of the inequality \eqref{eq:lem411} and the Cauchy-Schwarz inequality leads to
\begin{align*}
  I&\le    \sum_{i=0}^Lw_i
    \left[\left(m\sigma_sv_h^i,v_h^i\right)_X\right]^\frac12
    \left[\sum_{l=0}^L\left(\sigma_s w_l g(\cdot,{\bm{\omega}}_l
    \cdot{\bm{\omega}}_i)w_h^l,w_h^l\right)_X\right]^\frac12\\
    &\le
    \left[\sum_{i=0}^Lw_i\left(m\sigma_sv_h^i,v_h^i\right)_X\right]^\frac12
    \left[\sum_{i=0}^Lw_i\sum_{l=0}^L\left(\sigma_s w_l g(\cdot,{\bm{\omega}}_l
    \cdot{\bm{\omega}}_i)w_h^l,w_h^l\right)_X\right]^\frac12\nonumber\\
    &\le
    \left[\sum_{i=0}^Lw_i\left(m\sigma_sv_h^i,v_h^l\right)_X\right]^\frac12
    \left[\sum_{l=0}^Lw_l\left(m\sigma_sw_h^i,w_h^l\right)_X\right]^\frac12,
\end{align*}
which completes the proof of the lemma.
\end{proof}

For any $\bm{v}_h\in\bm{W}_h$, we define a norm $|||\cdot|||$ by
\begin{align*}
|||\bm{v}_h|||^2&=\sum_{l=0}^Lw_l\sum_{K\in T_h}c_0'\|v^l_h\|_{0,K}^2
  +\sum_{l=0}^Lw_l\sum_{\partial^l_+ K\subset  { \partial^l_+ X}}
 \left<v^l_-,v^l_-\bm{\omega}_l\cdot\bm{n}\right>_{\partial^l_+ K}\\
  &{}\quad+\delta\sum_{l=0}^Lw_l\sum_{K\in T_h}\|\bm{\omega}_l\cdot\nabla v_h^l\|_{0,K}^2
+\sum_{l=0}^Lw_l\sum_{K\in T_h}\left<[v^l_h],[v^l_h]
|\bm{\omega}_l\cdot \bm{n}|\right>_{\partial^l_- K}.
\end{align*}
We remark that $[v^l_h]= v^l_+$ on $\partial^l_- K\subset \partial^l_- X$, $l=0,\cdots,L$.

Then we prove a stability estimate for the method (\ref{DODSD:global2b})--(\ref{DODSD:global2e}).

\begin{lem}\label{lem:32}
  For sufficiently small $h$, we have
\[ |||\bm{v}_h|||^2\le 3a_h(\bm{v}_h,\bm{v}_h)\quad \forall \,\bm{v}_h\in \bm{W}_h. \]
\end{lem}
\begin{proof}
Noting that $\bm{\omega}_l$ is a constant vector, we get from  the Green formula that
 \begin{equation}\label{eq:3.25}
    \left(\bm{\omega}_l\cdot \nabla v_h,v_h\right)_K=-\left(v_h,\bm{\omega}_l
\cdot \nabla v_h\right)_K+\left<v_h,v_h\bm{\omega}_l\cdot \bm{n}\right>_{\partial K}
=\frac12\left<v_h,v_h\bm{\omega}_l\cdot \bm{n}\right>_{\partial K}.
 \end{equation}
Thus,
\[ a_h(\bm{v}_h,\bm{v}_h)=\sum_{l=0}^Lw_l\sum_{K\in T_h}\delta\,{ \|{\bm{\omega}}_l
 \cdot\nabla v_h^l\|_{0,K}^2}
  +\sum_{l=0}^Lw_l\sum_{K\in T_h}\left(\sigma_tv_h^l,v_h^l\right)_{K}+I_1+I_2+I_3+I_4,\]
where
\begin{align*}
I_1&=\sum_{l=0}^Lw_l\sum_{K\in T_h}\left(\delta\,\sigma_tv_h^l,\bm{\omega}_l\cdot\nabla v_h^l\right)_K,\\
I_2&=\sum_{l=0}^Lw_l\sum_{K\in T_h}\left(\frac{1}{2}
   \left<v_h^l,v_h^l\bm{\omega}_l\cdot\bm{n}\right>_{\partial K}
 +\left<{ [v_h^l]},v_+^l|{\bm{\omega}_l}\cdot \bm{n}|\right>_{\partial^l_- K}\right),\\
I_3&=-\sum_{l=0}^Lw_l\sum_{K\in T_h}\left(\sigma_s\sum_{i=0}^L w_i g(\cdot,{\bm{\omega}}_l
   \cdot{\bm{\omega}}_i)v_h^i,v_h^l\right)_K,\\
I_4&=-\sum_{l=0}^Lw_l\sum_{K\in T_h}\left(\sigma_s\sum_{i=0}^L w_i g(\cdot,{\bm{\omega}}_l
  \cdot{\bm{\omega}}_i)v_h^i,\delta {\bm{\omega}_l}\cdot \nabla v_h^l\right)_K.
\end{align*}

By the Cauchy-Schwarz inequality, we get
\begin{align*}
\left|I_1\right|
&\le\left[\sum_{l=0}^Lw_l\sum_{K\in T_h}\left(\delta\,\sigma_tv_h^l,\sigma_tv_h^l\right)_K\right]^\frac12
   \left[\sum_{l=0}^Lw_l\sum_{K\in T_h}\left(\delta\,\bm{\omega}_l\cdot\nabla v_h^l,
    \bm{\omega}_l\cdot\nabla v_h^l\right)_K\right]^\frac12\\
&\le  \frac12\delta\sum_{l=0}^Lw_l\sum_{K\in T_h}\left(\sigma_tv_h^l,\sigma_tv_h^l\right)_K
   +\frac12\sum_{l=0}^Lw_l\sum_{K\in T_h}\left(\delta\,\bm{\omega}_l\cdot\nabla v_h^l,
  \bm{\omega}_l\cdot\nabla v_h^l\right)_K.
\end{align*}

A simple calculation yields
\begin{align*}
I_2 & =\sum_{l=0}^Lw_l\sum_{K\in T_h}\left(
-\frac{1}{2}\left<v^l_+,v^l_+\left|\bm{\omega}_l\cdot\bm{n}\right|\right>_{\partial^l_- K}
+\frac{1}{2}\left<v^l_-,v^l_-\bm{\omega}_l\cdot\bm{n}\right>_{\partial^l_+ K}
+\left<[v_h^l],v^l_+|{\bm{\omega}}\cdot \bm{n}|\right>_{\partial^l_- K}\right)\\
& =\sum_{l=0}^Lw_l\sum_{K\in T_h}\left(
-\frac{1}{2}\left<v^l_+,v^l_+\left|\bm{\omega}_l\cdot\bm{n}\right|\right>_{\partial^l_- K}
+\frac12\left<v^l_-,v^l_-\left|\bm{\omega}_l\cdot\bm{n}\right|\right>_{\partial^l_- K}
+\left<[v_h^l],v^l_+|{\bm{\omega}}\cdot \bm{n}|\right>_{\partial^l_- K}\right)\\
&\quad{}+\sum_{l=0}^Lw_l\sum_{\partial^l_+ K\subset  { \partial^l_+ X}}\frac12\left<v^l_-,
v^l_-\left|\bm{\omega}_l\cdot\bm{n}\right|\right>_{\partial^l_+ K}\\
&=\sum_{l=0}^Lw_l\sum_{K\in T_h}\left(
\frac{1}{2}\left<[v_h^l],[v_h^l]\left|\bm{\omega}_l\cdot\bm{n}\right|\right>_{\partial^l_- K}\right)
+\sum_{l=0}^Lw_l\sum_{\partial^l_+ K\subset  { \partial^l_+ X}}\frac12\left<v^l_-,
v^l_-\left|\bm{\omega}_l\cdot\bm{n}\right|\right>_{\partial^l_+ K},
\end{align*}
where the condition that $v^l_-=0$ on $\partial^l_- K\subset \partial^l_- X$ is used.

Using Lemma \ref{lem:31}, we get
\[   \left| I_3\right|\le\sum_{l=0}^Lw_l\sum_{K\in T_h}\left(m\sigma_s v_h^l,v_h^l\right)_K\]
and
\begin{align*}
\left|I_4\right|
&\le \left[ \sum_{l=0}^Lw_l\sum_{K\in T_h}\left(m\sigma_s v_h^l,v_h^l\right)_K\right]^\frac12
 \left[ \sum_{l=0}^Lw_l\sum_{K\in T_h}\left(m\sigma_s \delta^2\bm{\omega}_l\cdot\nabla v_h^l,
  \bm{\omega}_l\cdot\nabla v_h^l\right)_K\right]^\frac12\\
&\le \frac12\delta^\frac23\sum_{l=0}^Lw_l\sum_{K\in T_h}\left(m\sigma_s v_h^l,v_h^l\right)_K
+\frac12\delta^\frac43 \sum_{l=0}^Lw_l\sum_{K\in T_h}\left(m\sigma_s \bm{\omega}_l
\cdot\nabla v_h^l,\bm{\omega}_l\cdot\nabla v_h^l\right)_K.
\end{align*}

Combining the above inequalities, we have
\begin{align*}
a_h(\bm{v}_h,\bm{v}_h)
&\ge\sum_{l=0}^Lw_l\sum_{K\in T_h}\left(\frac12\delta\left( 1-\delta^\frac13m\sigma_s\right)
{\bm{\omega}}_l\cdot \nabla v_h^l,{\bm{\omega}}_l\cdot \nabla v_h^l\right)_{K}\\
&{}\quad+\sum_{l=0}^Lw_l\sum_{K\in T_h}\left(\left(\sigma_t-\frac12\delta\,\sigma_t^2
  -(1+\delta^\frac23)m\sigma_s\right)v_h^l,v_h^l\right)_{K}\\
&{}\quad +\sum_{l=0}^Lw_l\sum_{K\in T_h}\left(\frac{1}{2}
  \left<[v_h^l],[v_h^l]\left|\bm{\omega}_l\cdot\bm{n}\right|\right>_{\partial ^l_-K}\right)
   +\sum_{l=0}^Lw_l\sum_{\partial^l_+ K\in { \partial^l_+ X}}\frac12\left<v^l_-,
  v^l_-\left|\bm{\omega}_l\cdot\bm{n}\right|\right>_{\partial^l_+ K}.
\end{align*}
Then the lemma can be obtained by taking a sufficiently small $h$.
\end{proof}

The unique solvability of the method \eqref{DODSD:global2b}--\eqref{DODSD:global2e}
is a direct consequence of the above lemma.

\begin{thm}
  For sufficiently small $h$, the DODSD method \eqref{DODSD_Global} has a unique  solution.
\end{thm}

\subsection{Error estimate}

For any $K\in T_h$, let $P_K$ be the orthogonal projection operator from $L^2(K)$ onto $P_k(K)$.
Then by the scaling argument and the trace theorem we can easily obtain the following result
(cf.~\cite{book:BS2008}).

\begin{lem}\label{lem:34}
For all $v\in H^{1+r}(K)$ with $r>0$ and $K\in T_h$, we have
\[   \|v-P_Kv\|_{0,K}+h_K\|v-P_Kv\|_{1,K}+h^\frac12_K\|v-P_Kv\|_{0,\partial K}
  \le C h_K^{1+\min\{r,k\}}\|v\|_{r+1,K}.\]
\end{lem}

For later analysis, we make a regularity assumption:
\begin{eqnarray} \label{as:2}
  \text{ for some } r>0, \ u^l\in H^{1+r}(X)\cap C(\overline{X}),\ 0\le l \le L.
\end{eqnarray}

\begin{thm}
Let $\{u^l\}$ and $\bm{u}_h$ be the solutions of \eqref{eq:DO} and
\eqref{DODSD:global2b}--\eqref{DODSD:global2e}, respectively. Under assumptions
\eqref{as:1} and \eqref{as:2}, we have, for all sufficiently small $h$,
\begin{equation}\label{eqn:es}
|||\{u^l\}-\bm{u}_h|||\le C_1 h^{\min\{r,k\}+\frac12}\left(\sum_{l=0}^L\|u^l\|_{r+1,X}\right)^\frac12.
\end{equation}
\end{thm}

\begin{proof}
By the regularity assumption  \eqref{as:2}, we have
  \begin{equation}
    a_h\left(\{u^l\},\{v^l_h\}\right)=0\quad \forall\, \{v^l_h\}\in \bm{V}_h.
  \end{equation}
Subtracting the above equality from   \eqref{DODSD:global2b}, we obtain the Galerkin orthogonality
  \begin{equation}\label{Gorth}
    a_h\left(\{u^l\}-\bm{u}_h,\{v^l_h\}\right)=0\quad \forall\, \{v^l_h\}\in \bm{V}_h.
  \end{equation}

Let $P_h$ denote the $L^2$-orthogonal operator onto $V^l_h$, $0\le l\le L$, in an elementwise way,
i.e., for $v\in L^2(X)$, let
\[ P_hv|_K:=P_Kv\quad \forall\, K\in T_h. \]
Set
\[   \eta^l=u^l-P_h u^l,\quad \xi^l=P_hu^l-u^l_h, \text{ and } e^l=u^l-u^l_h. \]
Note that $e^l_-|_{\partial^l_- K}=0$ for each $\partial^l_- K\subset \partial^l_- X$.

From Lemma \ref{lem:32} and the Galerkin orthogonality \eqref{Gorth}, we have
\begin{eqnarray}\label{eq:3b}
  |||\{e^l\}|||^2
  \le  3 a_h\left(\{e^l\},\{e^l\}\right)
  =  3 a_h\left(\{e^l\},\{\eta^l\}\right).
\end{eqnarray}

On the other hand, by the definition of the bilinear form $a_h(\cdot,\cdot)$, we have
\begin{equation}
  a_h\left(\{e^l\},\{\eta^l\}\right)=I_1+I_2+I_3+I_4+I_5,
\label{eq:4.29}
\end{equation}
where
\begin{align*}
I_1&=\sum_{l=0}^Lw_l\sum_{K\in T_h}\left({\bm{\omega}}_l\cdot \nabla e^l,\eta^l\right)_K,\\
I_2&=\sum_{l=0}^Lw_l\sum_{K\in T_h}\delta\left({\bm{\omega}}_l\cdot \nabla e^l,{\bm{\omega}}_l
  \cdot \nabla \eta^l\right)_K,\\
I_3&=\sum_{l=0}^Lw_l\sum_{K\in T_h}\left(\sigma_t e^l,\eta^l
  +\delta\, {\bm{\omega}_l}\cdot \nabla \eta^l\right)_K\\
I_4&= \sum_{l=0}^Lw_l\sum_{K\in T_h}\left(\sigma_s\sum_{i=0}^L w_i g(\bm{x},{\bm{\omega}}_l
 \cdot{\bm{\omega}}_i)e^i,\eta^l+\delta \,{\bm{\omega}_l}\cdot \nabla \eta^l\right)_K,\\
I_5&=\sum_{l=0}^Lw_l\sum_{K\in T_h}\left<{ [e^l]},
 \eta^l_+|{\bm{\omega}_l}\cdot \bm{n}|\right>_{\partial^l_- K}.
\end{align*}

By using the Cauchy-Schwarz inequality, Young's inequality, and Lemma \ref{lem:34}, we get
\begin{align} \label{eq:4.30}
|I_1|&\le \sum_{l=0}^Lw_l\sum_{K\in T_h}\|{\bm{\omega}}_l\cdot \nabla e^l\|_{0,K}\|\eta^l\|_{0,K}\\
 &\le \sum_{l=0}^Lw_l\sum_{K\in T_h}Ch_K^{1+\min\{r,k\}}\|{\bm{\omega}}_l
  \cdot \nabla e^l\|_{0,K}\|u^l\|_{r+1,K}\nonumber\\
&\le\sum_{l=0}^Lw_l\sum_{K\in T_h}\left(\frac16\delta\,\|{\bm{\omega}}_l
\cdot \nabla e^l\|^2_{0,K}+Ch_K^{2+2\min\{r,k\}}\delta^{-1}\|u^l\|_{r+1,K}^2\right),\nonumber
\end{align}
\begin{align}\label{eq:4.31}
  |I_2|&\le \sum_{l=0}^Lw_l\sum_{K\in T_h}\left(\delta\,\|{\bm{\omega}}_l
\cdot \nabla e^l\|_{0,K}\|{\bm{\omega}}_l\cdot \nabla \eta^l\|_{0,K}\right)\\
&\le \sum_{l=0}^Lw_l\sum_{K\in T_h}\delta\, h^{\min\{r,k\}}_K\|{\bm{\omega}}_l
\cdot \nabla e^l\|_{0,K}\|u^l\|_{r+1,K}\nonumber\\
&\le \sum_{l=0}^Lw_l\sum_{K\in T_h}\left(\frac16\delta\,\|{\bm{\omega}}_l
\cdot \nabla e^l\|_{0,K}^2+ C h^{2\min\{r,k\}}_K\delta\,\|u^l\|_{r+1,K}^2\right),\nonumber
\end{align}
\begin{align}\label{eq:4.32}
 |I_3|&\le \sum_{l=0}^Lw_l\sum_{K\in T_h}\left(\|\sigma_t e^l\|_{0,K}
  \|\eta^l+\delta\,{\bm{\omega}}_l\cdot \nabla \eta^l\|_{0,K}\right)\\
&\le \sum_{l=0}^Lw_l\sum_{K\in T_h}\left(\frac16c_0'\|e^l\|_{0,K}^2+C
 \left(\|\eta^l\|_{0,K}^2+\delta\,\|{\bm{\omega}}_l\cdot \nabla \eta^l\|_{0,K}^2\right)\right)\nonumber\\
&\le\sum_{l=0}^Lw_l\sum_{K\in T_h}\left(\frac16c_0'\|e^l\|_{0,K}^2
  +C\left(h_K^{2+2\min\{r,k\}}\|u^l\|_{r+1,K}^2+\delta\, h_K^{2\min\{r,k\}}\|u^l\|^2_{r+1,K}\right)\right),
\nonumber
\end{align}
\begin{align}\label{eq:4.33}
 |I_4|&\le C\left[\sum_{l=0}^Lw_l\sum_{K\in T_h}\|e^l\|^2_{0,K}\right]^\frac12
 \left[\sum_{l=0}^Lw_l\sum_{K\in T_h}\|\eta^l+\delta\,{\bm{\omega}}_l\cdot \nabla \eta^l\|^2_{0,K}\right]^\frac12\\
 &\le \sum_{l=0}^Lw_l\sum_{K\in T_h}\frac16 c_0'\|e^l\|^2_{0,K}+ C\sum_{l=0}^Lw_l
 \sum_{K\in T_h}\left(\|\eta^l\|_{0,K}^2+\delta\,\|{\bm{\omega}}_l\cdot \nabla \eta^l\|^2_{0,K}\right)\nonumber\\
 &\le \sum_{l=0}^Lw_l\sum_{K\in T_h}\frac16 c_0'\|e^l\|^2_{0,K}+
 C\sum_{l=0}^Lw_l\sum_{K\in T_h}\left(h_K^{2+2\min\{r,k\}}\|u^l\|_{r+1,K}^2
  +\delta\, h_K^{2\min\{r,k\}}\|u^l\|_{r+1,K}^2\right)\nonumber
\end{align}
and
\begin{align}\label{eq:3e}
|I_5|&\le\sum_{l=0}^Lw_l\sum_{K\in T_h}\left(\frac16\left<[e^l],[e^l]|{\bm{\omega}}_l
\cdot \bm{n}|\right>_{\partial ^l_- K}+C\left<\eta^l_+,\eta^l_+|{\bm{\omega}}_l
\cdot \bm{n}|\right>_{\partial^l_- K}\right)\\
&\le \sum_{l=0}^Lw_l\sum_{K\in T_h}\left(\frac16\left<[e^l],[e^l]|{\bm{\omega}}_l
\cdot \bm{n}|\right>_{\partial^l_- K}+Ch_K^{1+2\min\{r,k\}}\|u^l\|_{r+1,K}^2\right). \nonumber
\end{align}

Combining (\ref{eq:3b})--(\ref{eq:3e}), we obtain
\begin{equation}
  |||\{e^l\}|||^2 \le \frac12|||\{e^l\}|||^2 +C h_K^{1+2\min\{r,k\}}
\sum_{l=0}^Lw_l\sum_{K\in T_h}\|u^l\|_{r+1,K}^2,
\end{equation}
where $\delta=\bar{c}\,h$ is used. Thus we complete the proof of this theorem.
\end{proof}

\begin{remark}
Note that $\|\bm{\omega}_l\cdot\nabla (u^l-u_h^l)\|_{0,K}$ is included in the norm
$|||\{u^l\}-\bm{u}_h|||$, therefore \eqref{eqn:es} also gives a stability estimate for
$\|\bm{\omega}_l\cdot\nabla (u^l-u_h^l)\|_{0,K}$ in terms of $\|u^l\|_{r+1,X}$.
We remark that this estimate was not established for the DODG approximation solution of the RTE,
cf.\ Theorem 4.6 in  \cite{HHE2010}.
\end{remark}

Error estimates between the solution $u$ to the RTE and the solution $\{u^l\}$ to the
semi-discretized equation (\ref{eq:DO}) have been proved in \cite{HHE2010}.

\begin{thm}\label{thm:4.7}
Let  $\{u^l\}$ and $u$ be the solutions of \eqref{eq:DO}
and \eqref{eq:rte}--\eqref{eq:rtee}, respectively.
In 3D, if the regularity assumption \eqref{as:2} holds, then we have
\begin{equation}
\left(\sum_{l=0}^lw_l\sum_{K\in T_h}\left\|u^l(\cdot)-u(\cdot,\bm{\omega}_l)
\right\|_{0,K}^2\right)^\frac12\le C_2n^{-r-1}\left(\int_X\|u(\cdot,\cdot)\|_{r+1,\Omega}^2
d\bm{x}\right)^\frac12,
\end{equation}
where $C_1$ is positive constant depending on $r$ and the phase function $g$.
\end{thm}

Similarly, we can obtain the following theorem.

\begin{thm}\label{thm:4.8}
Let  $\{u^l\}$ and $u$ be the solutions of \eqref{eq:DO} and \eqref{eq:rte}--\eqref{eq:rtee},
respectively. In 2D, if the solution $u$
to RTE  \eqref{eq:rte}--\eqref{eq:rtee} is in $L^2(X, C^2(\Omega))$
and there exists a positive constant $C$ such that
\begin{equation}
  \sup_{\bm{x}\in X,\ \bm{\omega}\in \Omega}\|g''(\bm{x},\bm{\omega}\cdot )\|_{0,\infty,\Omega} \le C,
\end{equation}
where $g''(\bm{x},t)=\frac{\partial^2 g(\bm{x},t)}{\partial t^2}$, then we have
\begin{equation}
    \left(\sum_{l=0}^lw_l\sum_{K\in T_h}\left\|u^l(\cdot)-u(\cdot,\bm{\omega}_l)
\right\|_{0,K}^2\right)^\frac12= O(h_\theta^{2}),
\end{equation}
when $h_\theta$ is sufficiently small.
\end{thm}

Combining the above three theorems, we obtain the following results.

\begin{thm}
Let $\bm{u}_h$ and $u$ be the solutions of
\eqref{DODSD:global2b}--\eqref{DODSD:global2e} and \eqref{eq:rte}--\eqref{eq:rtee}, respectively.
Under the assumption of Theorem \ref{thm:4.7}, we have
\begin{align}\label{eq:4.39}
\left(\sum_{l=0}^lw_l\sum_{K\in T_h}\left\|u^l_h(\cdot)-u(\cdot,\bm{\omega}_l)
\right\|_{0,K}^2\right)^\frac12
&\le C_1 h^{\min\{r,k\}+\frac12}\left(\sum_{l=0}^Lw_l\sum_{K\in T_h}\|u^l\|_{r+1,K}^2\right)^\frac12\\
&{}\quad + C_2n^{-r-1}\left(\int_X\|u(\cdot,\cdot)\|_{r+1,\Omega}^2d\bm{x}\right)^\frac12,\nonumber
  \end{align}
  when  $h$ is  sufficiently small.
\end{thm}

\begin{thm}
Let $\bm{u}_h$ and $u$ be the solutions of
\eqref{DODSD:global2b}--\eqref{DODSD:global2e} and \eqref{eq:rte}--\eqref{eq:rtee}, respectively.
Under the assumption of Theorem \ref{thm:4.8}, we have
\begin{equation}
\left(\sum_{l=0}^lw_l\sum_{K\in T_h}\left\|u^l_h(\cdot)-u(\cdot,\bm{\omega}_l)
\right\|_{0,K}^2\right)^\frac12    \le  C_1 h^{\min\{r,k\}+\frac12}\left(\sum_{l=0}^Lw_l
\sum_{K\in T_h}\|u^l\|_{r+1,K}^2\right)^\frac12  +O(h_\theta^2),
  \end{equation}
  when  $h$ is  sufficiently small.
\end{thm}

\section{Numerical experiments}\label{sec:NE}
In this section, we present some numerical examples of the discrete-ordinate
discontinuous-streamline diffusion method for the radiative transfer equation
\eqref{eq:rte}--\eqref{eq:rtee} in the 2D case.
The main purpose is to illustrate the convergence performance of the proposed DODSD method
and the effect of the added diffusion parameter.

\subsection{Implementation}

First, we briefly describe the implementation of the DODSD method.
For a mesh shown in Figure \ref{fig21}, the DODSD method can be carried out for one direction $\omega$
in the following order:
\begin{figure}[htb]
\begin{center}
  \includegraphics[trim=0 318 0 260,clip,width=5in]{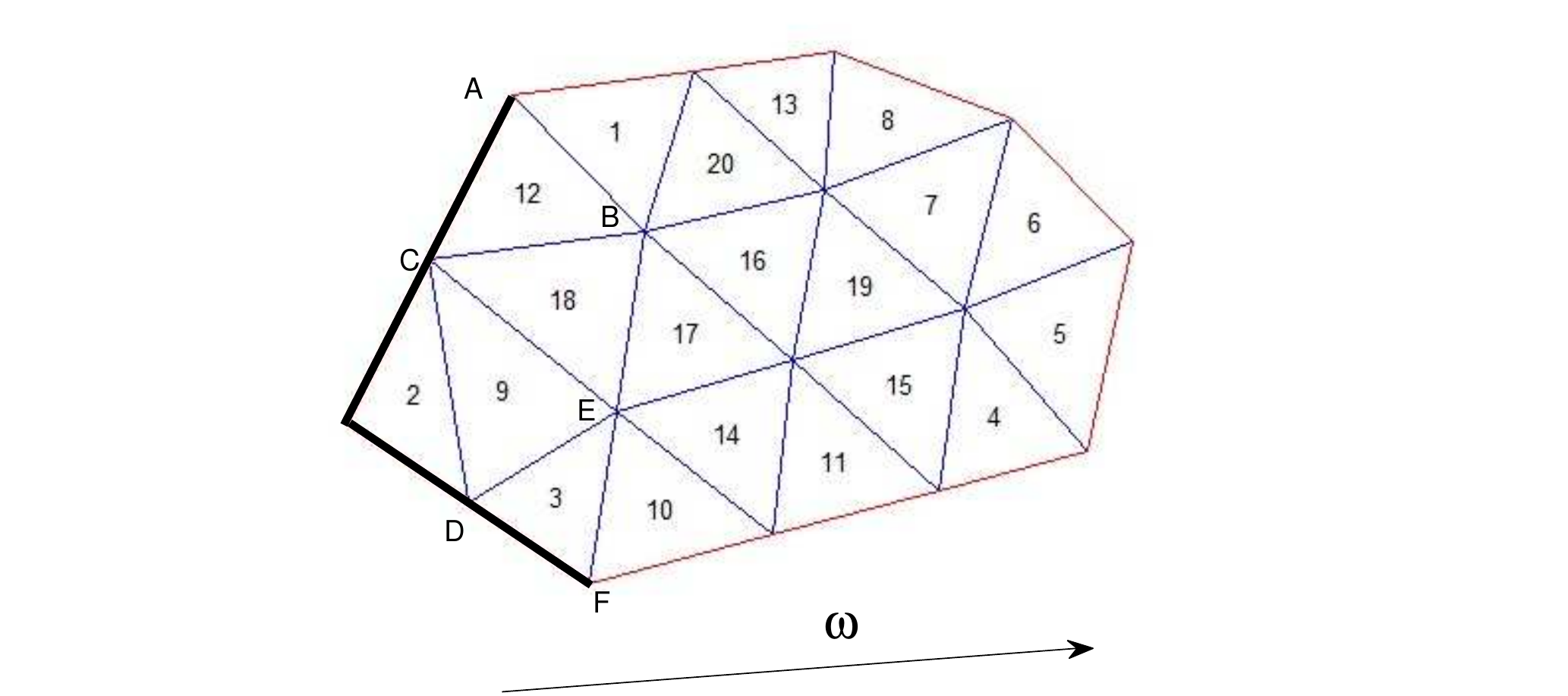}
\end{center}
\caption{A example of $T_h$ in 2D }\label{fig21}
\end{figure}

{\bf Step 1.} Denote by $T^{(1)}_h$  the elements for which all incoming boundary
$\partial^l_- K \subset \partial_-^l X$. In Figure \ref{fig21},
$T^{(1)}_h=\{K_i\colon i=2,3\}$.
We first compute $u^l_h$ for $K \in T^{(1)}_h$.

{\bf Step 2.} For $T_h\backslash T_h^{(1)}$, let $\partial_-^{l,1} X
=\{e\subset\partial^l_- K\colon K \in T_h\backslash T^{(1)}_h, \mbox{ and } u^l_+|_{e}
\mbox{ has been computed or given}\}$ denote its incoming edge.  In  Figure \ref{fig21},
$\partial_-^{l,1}$ is the broken line $\overline{ACDEF}$. Similarly, we define the set
$T^{(2)}_h$ and  compute $u^l_h$ for $K \in T^{(2)}_h$.
In Figure \ref{fig21}, $T^{(2)}_h=\{K_i\colon i=9,10\}$.

{\bf Step 3.} Repeating step 2, we obtain the non-ovrelapping decomposition
$T_h=T^{(1)}_h\cup T^{(2)}_h\cup \cdots \cup T^{(s)}_h$.
The computation should follow this sequence; that is, start the computation with the elements in
$T^{(1)}_h$ and end with the elements in $T^{(s)}_h$.

In the above procedure, the unknown function $u^l_h$ on each element $K$ is computed by following
the source iteration scheme of \eqref{DODSD}, that is, with an initial guess $u^{l,0}_h\in P_k(K)$,
$0\le l\le L$, for $j=1,2,\cdots$, we seek $u^{l,j}_h\in P_k(K)$, $0\le l\le L$, such that
\begin{align}\label{DODSD_SIb}
 & \left({\bm{\omega}}_l\cdot \nabla u^{l,j}_h +\sigma_tu^{l,j}_h,v^l_h
+\delta\, {\bm{\omega}_l}\cdot \nabla v_h^{l,j}\right)_K
 +\left<[u_h^{l,j}],v^l_+|{\bm{\omega}_l}\cdot \bm{n}|\right>_{\partial^l_- K}\\
 &\qquad =\left(\sigma_s\sum_{i=0}^L w_i g(\cdot,{\bm{\omega}}_l\cdot{\bm{\omega}}_i)
u_h^{i,j-1}+f_l,v^l_h+\delta \,{\bm{\omega}_l}\cdot \nabla v^l_h\right)_K\quad \forall\, v^l_h\in P_k(K)
\nonumber
\end{align}
with
\begin{equation}\label{DODSD_SIe}
  u_-^{l,j}|_{\partial^l_- K}=0,\quad \partial^l_- K\subset \partial_-^l X.
\end{equation}
For any $K\in T_h$ and $v_h\in P_k(K)$, it is easy to prove that  there exists a positive constant $C_K$ satisfying
\begin{align*}
&C_K\left(\|v_h\|_{0,K}^2+\delta\,\|{\bm{\omega}}_l\cdot \nabla v_h \|^2_{0,K}
+\frac12\left<{ v^l_-},v^l_-\left|{\bm{\omega}_l}
\cdot \bm{n}\right|\right>_{\partial^l_+ K}\right)\\
&\qquad \le \left({\bm{\omega}}_l\cdot \nabla v_h +\sigma_tv_h,v_h
+\delta\, {\bm{\omega}_l}\cdot \nabla v_h\right)_K
+\left<{ v^l_+},v^l_+|{\bm{\omega}_l}\cdot \bm{n}|\right>_{\partial^l_- K}
\end{align*}
when $h$ is sufficiently small. Then the unique solvability of \eqref{DODSD_SIb}--\eqref{DODSD_SIe}
can be obtained by the above inequality and the Lax-Milgram lemma (see e.g.\ \cite{book:C}).

We perform the above procedure for all directions in one iteration step, and stop the iteration
if some stopping condition is met, and take $\{u^{l,j}_h\}$ as $\{u^l_h\}$.

\subsection{Numerical experiments} \label{sec:NR}

Let $X=(0,1)\times (0,1)$. We consider the following four examples of the radiative transfer equation
\eqref{eq:rte}--\eqref{eq:rtee}:

{\bf Example 1.} the H-G
phase function with $\eta=0.2$.

{\bf Example 2.} the H-G
phase function with $\eta=0.5$.

{\bf Example 3.} the H-G
phase function with $\eta=0.9$.

{\bf Example 4.} the phase function
\[   g(\bm{x},t)=\frac1{2\pi}\left(1+\frac{t}2 \right).  \]
For {\bf Example 1} - {\bf Example 3}, the true solution is
\begin{equation*}
 u(\bm{x},\bm{\omega})=\sin(\pi x_1)\sin(\pi x_2).
\end{equation*}
And for {\bf Example 4}, the true solution is
\begin{equation*}
  u(\bm{x},\bm{\omega})=e^{-ax_1-bx_2}\left(1+c \cos\theta\right),
\end{equation*}
with $a=b=\frac{\sigma_a}3$ and $c=\frac{\sigma_a}{\sigma_a+6\sigma_s}$.
We set the right hand function $f(\bm{x},\bm{\omega})$ to satisfy the radiative transfer equation.

Let  $T_0=T_{h_0}$ be an initial triangulation of $X$ with a mesh size $h_0\approx 0.1$.
Then we recursively generate a sequence of nested triangulations $T_l=T_{h_l}$, $l=1,2,3$,
by dividing each triangle in the previous mesh $T_{l-1}$ into four sub-triangles by connecting
the midpoints of the edges; $h_l={2^{-l}h_0}$.  Based on these meshes, the linear finite element spaces are
constructed and used in the spatial discretization.
For the angular discretization, we employ the composite trapezoidal rule \eqref{ctf} with
$h_\theta=\pi/10$, $\pi/20$, $\pi/30$ and $\pi/10$ for the above four examples respectively.

We shall use the DODSD method with $\delta=h_l$  to solve these examples.
To measure the difference between the true solution and its approximate solution,
we define the  quantity $|||u-u_h|||_h:=\left(\sum_{i=1}^4\left(|||u-u_h|||^{(i)}\right)^2\right)^\frac12$
with
\begin{eqnarray*}
  |||v|||^{(1)}=\bigg(\sum_{l=0}^Lw_l\sum_{K\in T_h}\|v^l\|_{0,K}^2\bigg)^\frac12,\qquad\quad\
  |||v|||^{(2)}=\bigg(\sum_{l=0}^Lw_l\sum_{\partial^l_+ K\subset  \partial^l_+ X}
\left< v^l_-,v^l_-\bm{\omega}_l\cdot\bm{n}\right>_{\partial^l_+ K}\bigg)^\frac12,\\
  |||v|||^{(3)}=\bigg(\sum_{l=0}^Lw_l\sum_{K\in T_h}
  h_K \|\bm{\omega}_l\cdot\nabla v^l\|_{0,K}^2\bigg)^\frac12,\quad
  |||v|||^{(4)}=\bigg(\sum_{l=0}^Lw_l\sum_{K\in T_h}\left< { [v^l]},[v^l]|\bm{\omega}_l
  \cdot \bm{n}|\right>_{\partial^l_- K}\bigg)^\frac12.
\end{eqnarray*}

\subsubsection{Numerical convergence rates}
In this subsection, we take $\sigma_t=10$, $\sigma_s=0.1$.
Errors for these four examples are reported in Table \ref{tables51}--Table \ref{tables54} and
Figure \ref{fig:s51}--Figure \ref{fig:s54}. For all these examples, we can see that
$|||u-u_h|||^{(i)}$, $i=1, 2, 3$, are approximately $O(h^2)$, and that $|||u-u_h|||^{(4)}\approx O(h^{1.5})$.
Since $|||u-u_h|||\approx |||u-u_h|||_h$, we can conclude that $|||u-u_h|||=O(h^{1.5})$ for all these examples,
which agrees with our theoretical error estimates.

\begin{table}[ht]
\begin{center}
\caption{Error for Example 1}\label{tables51}
\begin{tabular*}{0.8\textwidth}{@{\extracolsep{\fill} } c c c c c c}
\toprule[1pt]
$l$& $|||u-u_h|||^{(1)}$&$ |||u-u_h|||^{(2)}$ & $ |||u-u_h|||^{(3)}$ & $|||u-u_h|||^{(4)}$& $|||u-u_h|||_h$\\
\midrule[1pt]
$0$             &5.3989e-3  & 6.1012e-3  &7.6011e-2  &3.3398e-2&8.3424e-2\\
\midrule[0.5pt]
$1$             &1.3923e-3  & 1.6388e-3  &2.6936e-2  &1.2970e-2&2.9973e-2\\
\midrule[0.5pt]
$2$             &3.5459e-4  & 4.3395e-4  &9.5335e-3  &4.8451e-3&1.0709e-2\\
\midrule[0.5pt]
$3$             &8.9879e-5  & 1.1300e-4  &3.3722e-3  &1.7661e-3&3.8094e-3\\
\bottomrule[1pt]
\end{tabular*}
\end{center}
\end{table}

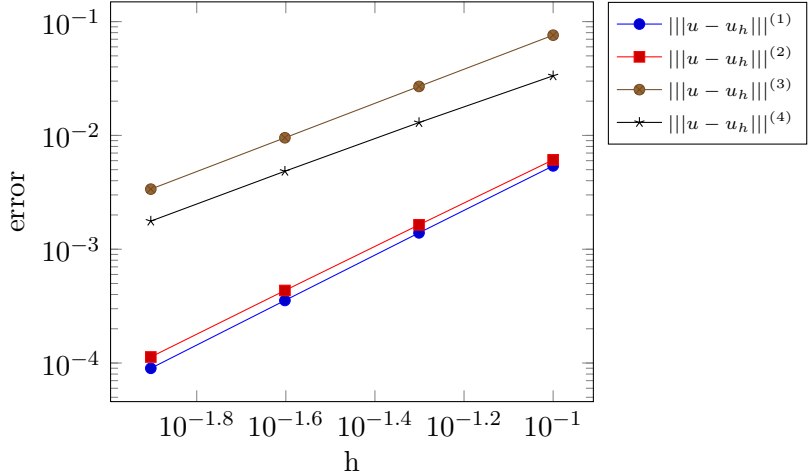
\begin{figure}[ht]
\centering
\pgfplotsset{width=8cm}
\begin{tikzpicture}
\begin{loglogaxis}[
xlabel={h},
ylabel=error,
legend pos=outer north east,
legend entries={\tiny $|||u-u_h|||^{(1)}$,\tiny $ |||u-u_h|||^{(2)}$, \tiny $|||u-u_h|||^{(3)}$, \tiny $|||u-u_h|||^{(4)}$}]
\addplot table[x=h,y=e1] {
h      e1
1.00e-1  5.3989e-003
5.00e-2  1.3923e-003
2.50e-2  3.5459e-004
1.25e-2  8.9879e-005
};
\addplot table[x=h,y=e2] {
h      e2
1.00e-1   6.1012e-003
5.00e-2   1.6388e-003
2.50e-2   4.3395e-004
1.25e-2   1.1300e-004
};
\draw (5.0e-2, 0.99e-2) --node[right]{$1.5$} (5.0e-2, 3.5e-3)
        -- node[below]{$1$}(2.5e-2, 3.5e-3) -- (5.0e-2, 0.99e-2);
\addplot table[x=h,y=e3] {
h      e3
1.00e-1    7.6011e-002
5.00e-2    2.6936e-002
2.50e-2    9.5335e-003
1.25e-2    3.3722e-003
};
\addplot table[x=h,y=e4] {
h       e1              e2            e3             e4
1.00e-1  5.3989e-003  6.1012e-003  7.6011e-002  3.3398e-002
5.00e-2  1.3923e-003  1.6388e-003  2.6936e-002  1.2970e-002
2.50e-2  3.5459e-004  4.3395e-004  9.5335e-003  4.8451e-003
1.25e-2  8.9879e-005  1.1300e-004  3.3722e-003  1.7661e-003
};
\draw (5.0e-2, 1.04e-3) --node[right]{$2.0$} (5.0e-2, 2.6e-4)
        -- node[below]{$1$}(2.5e-2, 2.6e-4) -- (5.0e-2, 1.04e-3);
\end{loglogaxis}
\end{tikzpicture}
\caption{Loglog convergence plot of $|||u-u_h|||^{(i)}$ ($i=1,2,3,4$)
vs.\ $h$ for Example 1}\label{fig:s51}
\end{figure}

\begin{table}[ht]
\begin{center}
\caption{Error for Example 2}\label{tables52}
\begin{tabular*}{0.8\textwidth}{@{\extracolsep{\fill} } c c c c c c}
\toprule[1pt]
$T_l$& $|||u-u_h|||^{(1)}$&$ |||u-u_h|||^{(2)}$ & $ |||u-u_h|||^{(3)}$ & $|||u-u_h|||^{(4)}$& $|||u-u_h|||_h$\\
\midrule[1pt]
$0$             &5.3951e-3  &6.1766e-3  & 7.6014e-2  & 3.3452e-2 &8.3453e-2\\
\midrule[0.5pt]
$1$             &1.3904e-3  &1.6591e-3  & 2.6937e-2  & 1.2994e-2 &2.9985e-2 \\
\midrule[0.5pt]
$2$             &3.5412e-4  &4.4085e-4  & 9.5337e-3  & 4.8564e-3 &1.0714e-2 \\
\midrule[0.5pt]
$3$             &8.9791e-5  &1.1504e-4  & 3.3723e-3  & 1.7711e-3 &3.8119e-3 \\
\bottomrule[1pt]
\end{tabular*}
\end{center}
\end{table}

\begin{figure}[ht] \vspace*{1cm}
\centering
\pgfplotsset{width=8cm}
\begin{tikzpicture}
\begin{loglogaxis}[
xlabel={h},
ylabel=error,
legend pos=outer north east,
legend entries={\tiny $|||u-u_h|||^{(1)}$,\tiny $ |||u-u_h|||^{(2)}$, \tiny $|||u-u_h|||^{(3)}$, \tiny $|||u-u_h|||^{(4)}$}]
\addplot table[x=h,y=e1] {
h      e1
1.00e-1   5.3951e-003
5.00e-2   1.3904e-003
2.50e-2   3.5412e-004
1.25e-2   8.9791e-005
};
\addplot table[x=h,y=e2] {
h      e2
1.00e-1   6.1766e-003
5.00e-2   1.6591e-003
2.50e-2   4.4085e-004
1.25e-2   1.1504e-004
};
\draw (5.0e-2, 0.99e-2) --node[right]{$1.5$} (5.0e-2, 3.5e-3)
        -- node[below]{$1$}(2.5e-2, 3.5e-3) -- (5.0e-2, 0.99e-2);

\addplot table[x=h,y=e3] {
h      e3
1.00e-1   7.6014e-002
5.00e-2   2.6937e-002
2.50e-2   9.5337e-003
1.25e-2   3.3723e-003
};
\addplot table[x=h,y=e4] {
h       e1              e2            e3             e4
1.00e-1  5.3951e-003  6.1766e-003  7.6014e-002  3.3452e-002
5.00e-2  1.3904e-003  1.6591e-003  2.6937e-002  1.2994e-002
2.50e-2  3.5412e-004  4.4085e-004  9.5337e-003  4.8564e-003
1.25e-2  8.9791e-005  1.1504e-004  3.3723e-003  1.7711e-003
};

\draw (5.0e-2, 1.04e-3) --node[right]{$2.0$} (5.0e-2, 2.6e-4)
        -- node[below]{$1$}(2.5e-2, 2.6e-4) -- (5.0e-2, 1.04e-3);
\end{loglogaxis}
\end{tikzpicture}
\caption{Loglog convergence plot of $|||u-u_h|||^{(i)}$ ($i=1,2,3,4$)
vs.\ $h$ for Example 2}\label{fig:s52}
\end{figure}
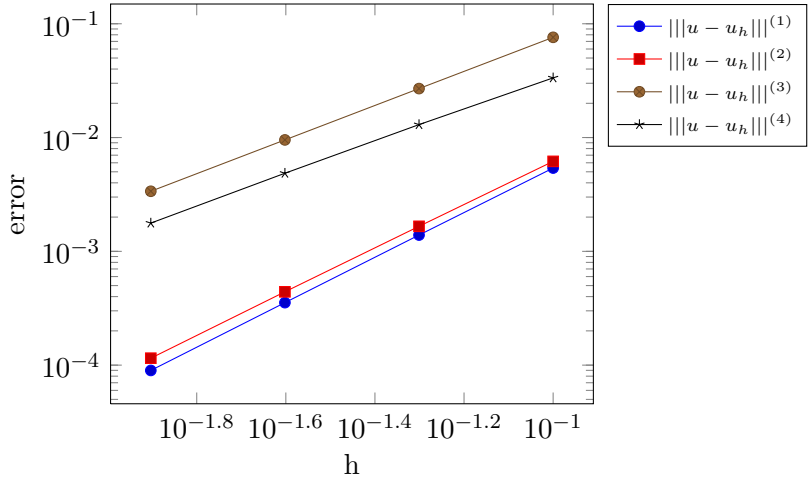

\begin{table}[ht]  \vspace*{1cm}
\begin{center}
\caption{Error for Example 3}\label{tables53}
\begin{tabular*}{0.8\textwidth}{@{\extracolsep{\fill} } c c c c c c }
\toprule[1pt]
$l$& $|||u-u_h|||^{(1)}$&$ |||u-u_h|||^{(2)}$ & $ |||u-u_h|||^{(3)}$ & $|||u-u_h|||^{(4)}$& $|||u-u_h|||_h$\\
\midrule[1pt]
$0$             &5.3969e-3  &6.1958e-3  &7.6013e-2 &3.3459e-2&8.3456e-2\\
\midrule[0.5pt]
$1$             &1.3910e-3  &1.6639e-3  &2.6936e-2 &1.2996e-2&2.9986e-2\\
\midrule[0.5pt]
$2$             &3.5655e-4  &4.4334e-4  &9.5332e-3 &4.8567e-3&1.0714e-2\\
\midrule[0.5pt]
$3$             &9.9651e-5  &1.1797e-4  &3.3719e-3 &1.7711e-3&3.8118e-3\\
\bottomrule[1pt]
\end{tabular*}
\end{center}
\end{table}

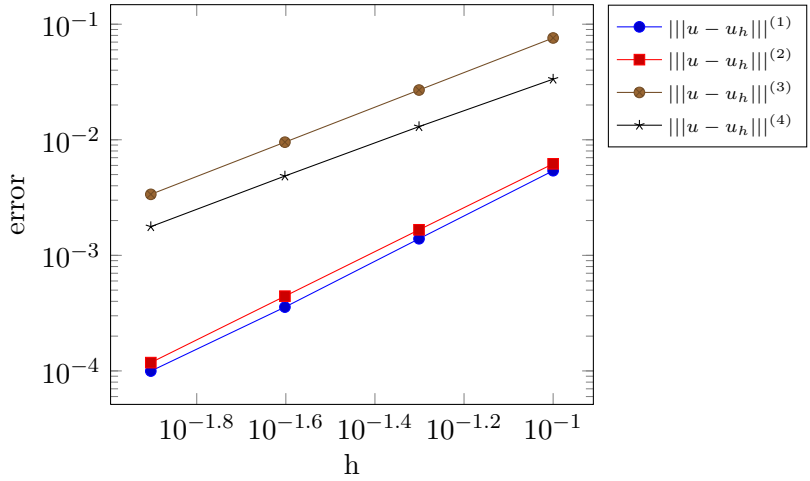
\begin{figure}[ht]
\centering
\pgfplotsset{width=8cm}
\begin{tikzpicture}
\begin{loglogaxis}[
xlabel={h},
ylabel=error,
legend pos=outer north east,
legend entries={\tiny $|||u-u_h|||^{(1)}$,\tiny $ |||u-u_h|||^{(2)}$, \tiny $|||u-u_h|||^{(3)}$, \tiny $|||u-u_h|||^{(4)}$}]
\addplot table[x=h,y=e1] {
h      e1
1.00e-1   5.3969e-003
5.00e-2   1.3910e-003
2.50e-2   3.5655e-004
1.25e-2   9.9651e-005
};
\addplot table[x=h,y=e2] {
h      e2
1.00e-1   6.1958e-003
5.00e-2   1.6639e-003
2.50e-2   4.4334e-004
1.25e-2   1.1797e-004
};
\draw (5.0e-2, 0.99e-2) --node[right]{$1.5$} (5.0e-2, 3.5e-3)
        -- node[below]{$1$}(2.5e-2, 3.5e-3) -- (5.0e-2, 0.99e-2);

\addplot table[x=h,y=e3] {
h      e3
1.00e-1  7.6013e-002
5.00e-2  2.6936e-002
2.50e-2  9.5332e-003
1.25e-2  3.3719e-003
};
\addplot table[x=h,y=e4] {
h       e1              e2            e3             e4
1.00e-1  5.3969e-003  6.1958e-003  7.6013e-002  3.3459e-002
5.00e-2  1.3910e-003  1.6639e-003  2.6936e-002  1.2996e-002
2.50e-2  3.5655e-004  4.4334e-004  9.5332e-003  4.8567e-003
1.25e-2  9.9651e-005  1.1797e-004  3.3719e-003  1.7711e-003
};
\draw (5.0e-2, 1.04e-3) --node[right]{$2.0$} (5.0e-2, 2.6e-4)
        -- node[below]{$1$}(2.5e-2, 2.6e-4) -- (5.0e-2, 1.04e-3);
\end{loglogaxis}
\end{tikzpicture}
\caption{Loglog convergence plot of $|||u-u_h|||^{(i)}$ ($i=1,2,3,4$)
vs.\ $h$ for Example 3}\label{fig:s53}
\end{figure}

\begin{table}[ht] \vspace*{0.5cm}
\begin{center}
\caption{Error for Example 4}\label{tables54}
\begin{tabular*}{0.8\textwidth}{@{\extracolsep{\fill} } c c c c c c }
\toprule[1pt]
$l$& $|||u-u_h|||^{(1)}$&$ |||u-u_h|||^{(2)}$ & $ |||u-u_h|||^{(3)}$ & $|||u-u_h|||^{(4)}$& $|||u-u_h|||_h$\\
\midrule[1pt]
$0$             &3.6110e-3 &2.6820e-3 &  3.1865e-2 & 1.2766e-2   &3.4620e-2 \\
\midrule[0.5pt]
$1$             &9.1999e-4 &7.8409e-4 &  1.1233e-2 & 5.1355e-3   &1.2410e-2\\
\midrule[0.5pt]
$2$             &2.3272e-4 &2.1117e-4 &  3.9834e-3 & 1.9544e-3   &4.4481e-3\\
\midrule[0.5pt]
$3$             &5.8632e-5 &5.5119e-5 &  1.4122e-3 & 7.1893e-4   &1.5867e-3 \\
\bottomrule[1pt]
\end{tabular*}
\end{center}
\end{table}

\begin{figure}[htp]\vspace*{1cm}
\centering
\pgfplotsset{width=8cm}
\begin{tikzpicture}
\begin{loglogaxis}[
xlabel={h},
ylabel=error,
legend pos=outer north east,
legend entries={\tiny $|||u-u_h|||^{(1)}$,\tiny $ |||u-u_h|||^{(2)}$, \tiny $|||u-u_h|||^{(3)}$, \tiny $|||u-u_h|||^{(4)}$}]
\addplot table[x=h,y=e1] {
h      e1
1.00e-1  3.6110e-003
5.00e-2  9.1999e-004
2.50e-2  2.3272e-004
1.25e-2  5.8632e-005
};
\addplot table[x=h,y=e2] {
h      e2
1.00e-1    2.6820e-003
5.00e-2    7.8409e-004
2.50e-2    2.1117e-004
1.25e-2    5.5119e-005
};
\draw (5e-2,     3.9598e-003) --node[right]{$1.5$} (5e-2, 1.4e-3)
        -- node[below]{$1$}(2.5e-2, 1.4e-3) -- (5e-2,     3.9598e-003);

\addplot table[x=h,y=e3] {
h      e3
1.00e-1    3.1865e-002
5.00e-2    1.1233e-002
2.50e-2    3.9834e-003
1.25e-2    1.4122e-003
};
\addplot table[x=h,y=e4] {
h       e1              e2            e3             e4
1.00e-1  3.6110e-003  2.6820e-003  3.1865e-002  1.2766e-002
5.00e-2  9.1999e-004  7.8409e-004  1.1233e-002  5.1355e-003
2.50e-2  2.3272e-004  2.1117e-004  3.9834e-003  1.9544e-003
1.25e-2  5.8632e-005  5.5119e-005  1.4122e-003  7.1893e-004
};
\draw (5e-2, 6e-4) --node[right]{$2.0$} (5e-2, 1.5e-4)
        -- node[below]{$1$}(2.5e-2, 1.5e-4) -- (5e-2, 6e-4);
\end{loglogaxis}
\end{tikzpicture}
\caption{Loglog convergence plot of $|||u-u_h|||^{(i)}$ ($i=1,2,3,4$)
vs.\ $h$ for Example 4}\label{fig:s54}
\end{figure}
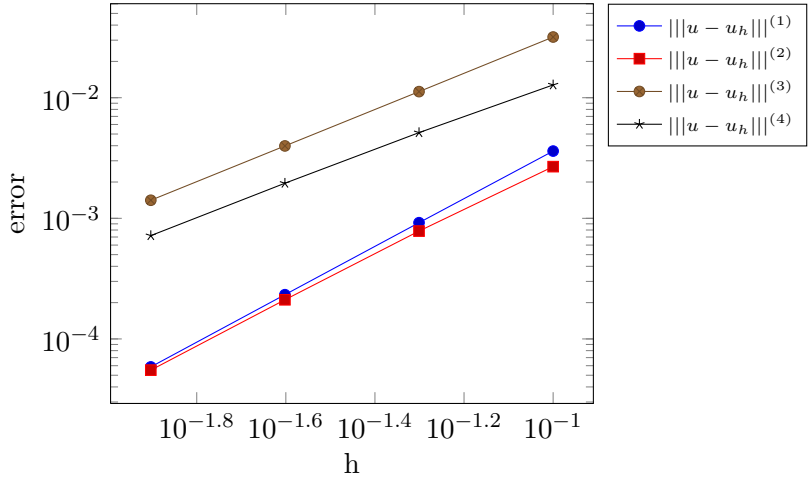


\subsubsection{Comparison with the DODG method}
In order to show the effects of adding the artificial diffusion term,
we report the error  of  the DODG method in norm $|||\cdot|||_h$ for the four examples in Table \ref{table52}.
The comparisons of the DODSD method and the DODG method are also shown in Figure \ref{fig:51}.
We observe that: 1) both  the DODSD method and the DODG method have the similar convergence rates;
2) the DODSD method can lead to some improvement of the accuracy in norm $|||\cdot|||_h$ compared to the DODG method.

\begin{table}[htb]
\begin{center}
\caption{Results of the DODG method}\label{table52}
\begin{tabular*}{0.6\textwidth}{@{\extracolsep{\fill} } c c c c c}
\toprule[1pt]
$l$           & {\bf Example 1}& {\bf Example 2}& {\bf Example 3}& {\bf Example 4}\\
\midrule[0.5pt]
$0$             &9.6214e-2 &9.6254e-2 &9.6265e-2 & 3.8551e-2 \\
\midrule[0.5pt]
$1$             &3.5124e-2 &3.5140e-2 &3.5141e-2 & 1.4422e-2 \\
\midrule[0.5pt]
$2$             &1.2668e-2 &1.2674e-2 &1.2673e-2 & 5.2930e-3 \\
\midrule[0.5pt]
$3$             &4.5303e-3 &4.5324e-3 &4.5316e-3 & 1.9121e-3 \\
\bottomrule[1pt]
\end{tabular*}
\end{center}
\end{table}
\begin{figure}[htb]
\centering
\pgfplotsset{width=6cm}
\begin{tikzpicture}
\begin{loglogaxis}[
title=Example 1,
xlabel={h},
ylabel=error,
]
\addplot table[x=h,y=e1] {
h      e1
1.00e-1  8.3424e-002
5.00e-2  2.9973e-002
2.50e-2  1.0709e-002
1.25e-2  3.8094e-003
};
\addplot table[x=h,y=e12] {
h      e12
1.00e-1  9.6214e-002
5.00e-2  3.5124e-002
2.50e-2  1.2668e-002
1.25e-2  4.5303e-003
};
\draw (5.0e-2,  2.2627e-002) --node[right]{$1.5$} (5.0e-2, 8e-3)
        -- node[below]{$1$}(2.5e-2, 8e-3) -- (5.0e-2,      2.2627e-002);
\end{loglogaxis}
\end{tikzpicture}
\begin{tikzpicture}
\begin{loglogaxis}[
title=Example 2,
xlabel={h},
ylabel=error,
]
\addplot table[x=h,y=e2] {
h      e2
1.00e-1  8.3453e-002
5.00e-2  2.9985e-002
2.50e-2  1.0714e-002
1.25e-2  3.8119e-003
};
\addplot table[x=h,y=e22] {
h      e22
1.00e-1   9.6254e-002	
5.00e-2   3.5140e-002	
2.50e-2   1.2674e-002	
1.25e-2   4.5324e-003	
};

\draw (5.0e-2,  2.2627e-002) --node[right]{$1.5$} (5.0e-2, 8e-3)
        -- node[below]{$1$}(2.5e-2, 8e-3) -- (5.0e-2,      2.2627e-002);
\end{loglogaxis}
\end{tikzpicture}
\begin{tikzpicture}
\begin{loglogaxis}[
title=Example 3,
xlabel={h},
ylabel=error,
]
\addplot table[x=h,y=e3] {
h      e3
1.00e-1 8.3456e-002
5.00e-2 2.9986e-002
2.50e-2 1.0714e-002
1.25e-2 3.8118e-003
};
\addplot table[x=h,y=e32] {
h      e32
1.00e-1  9.6265e-002
5.00e-2  3.5141e-002
2.50e-2  1.2673e-002
1.25e-2  4.5316e-003
};
\draw (5.0e-2,  2.2627e-002) --node[right]{$1.5$} (5.0e-2, 8e-3)
        -- node[below]{$1$}(2.5e-2, 8e-3) -- (5.0e-2,      2.2627e-002);
\end{loglogaxis}
\end{tikzpicture}
\begin{tikzpicture}
\begin{loglogaxis}[
title=Example 4,
xlabel={h},
ylabel=error,
]
\addplot table[x=h,y=e4] {
h      e4
1.00e-1   3.4620e-002
5.00e-2   1.2410e-002
2.50e-2   4.4481e-003
1.25e-2   1.5867e-003
};
\addplot table[x=h,y=e42] {
h      e42
1.00e-1   3.8551e-002	
5.00e-2   1.4422e-002	
2.50e-2   5.2930e-003	
1.25e-2   1.9121e-003
};
\draw (5.0e-2, 9.6167e-003) --node[right]{$1.5$} (5.0e-2, 3.4e-3)
        -- node[below]{$1$}(2.5e-2, 3.4e-3) -- (5.0e-2, 9.6167e-003);
\end{loglogaxis}
\end{tikzpicture}
\caption{Loglog convergence plot of $|||u-u_h|||_h$
vs.\ $h$ (red line: DODG method; blue line: DODSD method)}\label{fig:51}
\end{figure}
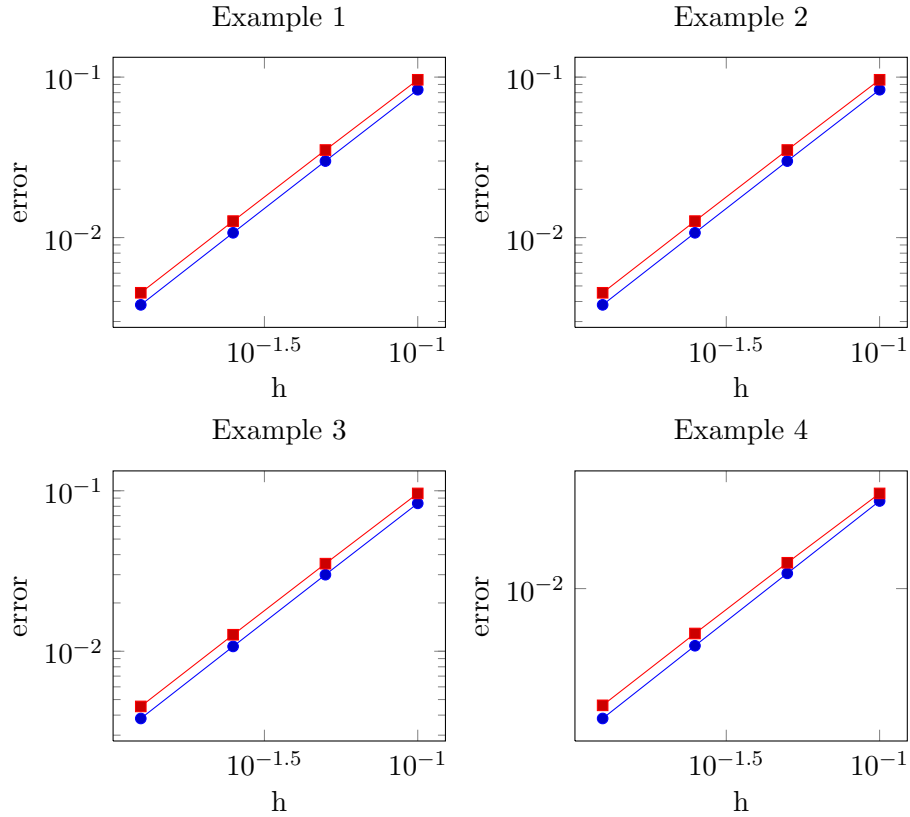

\section{Conclusion}\label{sec:conclusion}
In this paper, we present a discrete-ordinate discontinuous-streamline diffusion method for
solving the radiative transfer equation. This method applies the discrete ordinate technique to deal with
the integration term of the radiative transfer equation in the angular discretization, and employs the
discontinuous-streamline diffusion method for the spatial discretization. The stability property
and unique solvability of the discrete system are proved. Under suitable solution regularity assumptions,
error estimates for the numerical solutions are derived in a norm including the directional gradient.
Numerical results confirm the convergence behavior of the proposed method.

The main difference between the DODSD method and the DODG method is in the additional artificial diffusion term.
Our numerical experiments show that such a modification can improve the accuracy of numerical solutions in term
of $|||\cdot|||_h$ norm in comparison with the DODG method.
As for the effect of the artificial diffusion parameter $\delta$, we remark that it may reduce the error
$|||u-u_h|||^{(3)}$ and $|||u-u_h|||^{(4)}$ while increase the error $|||u-u_h|||^{(1)}$ and $|||u-u_h|||^{(2)}$.
Since $|||u-u_h|||^{(1)}$ and $|||u-u_h|||^{(2)}$ converge faster than $|||u-u_h|||^{(3)}$ and $|||u-u_h|||^{(4)}$,
the DODSD method with an appropriate $\delta$ is expected to be  more accurate in $|||\cdot|||$ norm in comparison
with the DODG method.



\bibliographystyle{abbrv}
\bibliography{DSD2RTE}

\end{document}